\documentclass[12pt,reqno]{article}

\usepackage{color}
\definecolor{webgreen}{rgb}{0,.5,0}
\definecolor{webbrown}{rgb}{.6,0,0}
\definecolor{redmaple}{rgb}{.73,.22,.05}
\definecolor{bluemaple}{rgb}{.29,0,.71}
\definecolor{extraproof}{rgb}{0,.5,.5}

\usepackage{amssymb}
\usepackage{amsmath}
\usepackage{amsthm}
\usepackage{amsfonts}
\usepackage{graphicx}
\usepackage{cancel}
\usepackage[colorlinks=true,
linkcolor=webgreen,
urlcolor=bluemaple,
filecolor=webbrown,
citecolor=webgreen]{hyperref}

\usepackage[symbol]{footmisc}

\usepackage{multirow}

\usepackage{enumitem,kantlipsum}

\usepackage{subfigure}

\begin{document}

\newtheorem{theorem}{Theorem}
\newtheorem*{theorem*}{Theorem}
\newtheorem{lemma}{Lemma}
\newtheorem{proposition}{Proposition}
\newtheorem*{proposition*}{Proposition}
\newtheorem{corollary}{Corollary}
\newtheorem*{corollary*}{Corollary}
\newtheorem{conjecture}{Conjecture}
\newtheorem{remark}{Remark}
\newtheorem{problem}{Problem}

\theoremstyle{definition}
\newtheorem{definition}{Definition}
\newtheorem*{definition*}{Definition}
\newtheorem{example}{Example}
\newtheorem*{example*}{Example}
\newtheorem{notation}{Notation}

\newcommand{\R}{{\mathbb R}}
\newcommand{\Q}{{\mathbb Q}}
\newcommand{\C}{{\mathbb C}}
\newcommand{\N}{{\mathbb N}}
\newcommand{\Z}{{\mathbb Z}}

\newcommand{\seqnum}[1]{\href{https://oeis.org/#1}{#1}}

\newcommand{\doi}[1]{\href{http://doi.org/#1}{DOI: #1}}

\newcommand{\arxiv}[1]{\href{https://arxiv.org/#1}{arXiv: #1}}

\makeatletter

\renewcommand\@makefntext[1]{%
\setlength\parindent{1em}%
\noindent
\mbox{$^\@thefnmark$~}{#1}}

\makeatother

\begin{center}
\vskip 1cm{\LARGE\bf 
A complete solution of the partition of

\vskip 0,3cm
a number into arithmetic progressions\footnote[2]{JP J. Algebra Number Theory Appl. 53(2) (2022), 109-122. \doi{10.17654/0972555522006}}}
\vskip 0,6cm
\large
F. Javier de Vega\\
King Juan Carlos University\\
Madrid, Spain\\
\href{mailto:javier.devega@urjc.es}{\tt javier.devega@urjc.es} \\
\end{center}

\vskip .2 in

\begin{abstract}
We solve the enumeration of the set $\textrm{AP}(n)$ of partitions of a positive integer $n$ in which the nondecreasing sequence of parts forms an arithmetic progression. In particular, we establish a formula for the number of nondecreasing arithmetic progressions of positive integers with sum $n$. We also present an explicit method to calculate all the partitions of $\textrm{AP}(n)$.
\end{abstract}



\section{Introduction}\label{sec:intro}
A partition of a positive integer $n$ is a nondecreasing sequence of positive integers whose sum is $n$. The summands are called parts of
the partition. We consider the problem of enumerating the set $\textrm{AP}(n)$ of partitions of $n$ in which the nondecreasing sequence of parts forms an arithmetic progression (AP), that is, the nondecreasing arithmetic progressions of positive integers with sum $n$.

A related work was made by Mason \cite{mason} who studied the representation of an integer as the sum of consecutive integers. Bush \cite{bush} extended Mason's results to integers in arithmetic progressions. An analytical approach was made by Leveque \cite{leveque}. A few authors have considered a combinatorial perspective of the problem of enumerating the set $\textrm{AP}(n)$ (see \cite{cook,munagi1,munagi2,nyblom}). The sequence $(|\textrm{AP}(n)|)_{n>0}$ occurs as sequence number \seqnum{A049988} in the Online Encyclopedia of Integer Sequences \cite{oeis}. 

Our paper proposes a novel way to study this problem based on \cite{devega}. The main idea is as follows: the usual divisors trivially solve the problem of the partition of a number into equal parts. Now, for each $k \in \Z$, we will consider a new product mapping ($\odot_{k}$) that will generate an arithmetic ($k$-\textit{arithmetic}) similar to the usual one. In this new arithmetic, the divisors of an integer $n$ will trivially solve the problem of the representation of $n$ as the sum of arithmetic progressions whose difference is $k$. We will prove the following theorem:
\begin{theorem} \label{theorem:1}
\noindent Given a positive integer $n$, let $\tau(n)$ denote the number of positive divisors of $n$, $\textrm{D}_{\textrm{E}}(n)$ denote the set of divisors of $n$ and $\textrm{D}_{\textrm{O}}(n)$ denote the set of divisors of $2n$ except the even divisors of $n$. Then the cardinality of the set $\textrm{AP}(n)$, denoted by $|\textrm{AP}(n)|$, is equal to
\begin{equation*}
\tau(n)+\sum_{\scriptstyle  d \ \in \ \textrm{D}_{\textrm{\tiny E}}(n) \atop \scriptstyle  1 \ < \ d \ \leq \ \sqrt{n}}\left\lfloor \frac{1}{2} (\left\lceil \frac{2n}{d(d-1)} \right\rceil -1)\right\rfloor + \sum_{\scriptstyle  d \ \in \ \textrm{D}_{ \textrm{\tiny O}}(n) \atop \scriptstyle 1 \ < \ d \ < \ \sqrt{2n}}\left\lfloor \frac{1}{2} \left\lceil \frac{2n}{d(d-1)} \right\rceil \right\rfloor.
\end{equation*}
\end{theorem}

We also present an explicit method to calculate all the partitions of $\textrm{AP}(n)$.

In the following section, we present a brief introduction to the methods used in \cite{devega}.

\section{Arithmetic progressions and the usual arithmetic} \label{sec:2}
The usual product $m\cdot n$ on $\Z$ can be viewed as the sum of $n$ terms of an arithmetic progression $(a_n)$ whose first term is $a_{1}=m-n+1$ and whose difference is $d=2$.

\begin{example} \label{example:1}
$6\cdot 3 = (6-3+1)+6+8=3\cdot 6=(3-6+1)+0+2+4+6+8$.
\end{example}

The previous example motivates the following definition.
\begin{definition}[$k$-\textit{arithmetic} product $\odot_{k}$] \label{def:1}
Given $m,k\in \Z$, for all positive integers $n$, we define the following expression 
\begin{equation*}
m\odot_{_{k}} n =(m-n+1)+(m-n+1+k)+\ldots+(m-n+1+k+\stackrel{(n-1)}{\ldots}+k)
\end{equation*}
as the $k$-\textit{arithmetic} product.
\end{definition}

This arithmetic progression can be added to obtain the following formula:
\begin{equation}
\label{eq:1}
m \odot_{_{k}} n =(m-n+1)\cdot n + \frac{n\cdot (n-1)\cdot k}{2}.
\end{equation}

We take \eqref{eq:1} as Definition \ref{def:1} and consider $n \in \Z$.

In connection with the above result, for each $k \in \Z$, the expression ``given a $k$-\textit{arithmetic}'' refers to the fact that we are going to work with integers, the sum, the new product and the usual order. This means that we are going to work on $\mathcal{Z}_k=\{ \Z , + ,\odot_{_{k}} , < \}$. Clearly, $\mathcal{Z}_2$, the $2$-\textit{arithmetic}, will be the usual arithmetic.

\begin{definition}[$k$-\textit{arithmetic} divisor] \label{def:2}
\noindent Given a $k$-\textit{arithmetic}, an integer $d>0$ is called a divisor of $a \ (arith \ k)$ if there exists some integer $b$ such that $a=b \odot_{_{k}} d$. We can write: $d \mid a \ (\textit{arith } k) \Leftrightarrow \exists b \in \Z \text{ such that } b\odot_{k} d=a$.
\end{definition}

In other words, $d$ is the number of terms of the summation that represents the $k$-\textit{arithmetic} product.

\begin{example}
\label{example:2}
Consider the following expression:
\begin{equation*}
9\odot_{3} 8=2+5+8+11+14+17+20+23=100.
\end{equation*}
The number of terms is $8$; hence, we can say that $8$ is a divisor of $100$ in $3$-\textit{arithmetic}, that is, $8$ is a divisor of $100$ (\textit{arith} $3$). Notably, a divisor is always a positive number, and the number $9$ indicates where we should start the summation. However, we cannot be sure that $9$ is a divisor of $100$ (\textit{arith} $3$).
\end{example}

To characterize the set of divisors, we define the $k$-\textit{arithmetic} quotient:

\begin{definition}[$k$-\textit{arithmetic} quotient $\oslash_{k}$] \label{def:3}
\noindent Let $a,b \in \Z$, $b\neq 0$. Given a $k$-\textit{arithmetic}, an integer $c$ is called a quotient of $a$ divided by $b$ (\textit{arith} $k$) if and only if $c\odot_{_{k}} b=a$. We write: $a \oslash_{k} b = c \Leftrightarrow c\odot_{_{k}} b=a$.
\end{definition}

Also, we can express the $k$-\textit{arithmetic} quotient with the usual one:
\begin{equation} \label{eq:2}
a \oslash_{k} b=\frac{a}{b}+(b-1)\cdot(1-\frac{k}{2}).
\end{equation}

We must consider $\oslash_{k}$ in the following manner. If we want to write $a$ as the sum of $b$ terms of an arithmetic progression, then the quotient will give us the place to start the summation. For instance, if we want to express $57$ as the sum of $6$ terms of an arithmetic progression whose difference is $3$, we can do: $57\oslash_{3} 6=57/6+5\cdot(1-3/2)=7$. Hence, $57=7 \odot_{3} 6$. The first term is $7-6+1=2$, and the solution is $2+5+8+11+14+17=57$.

\begin{corollary}
\label{corollary:1}
Let $b>0$, $b$ is a divisor of $a$ (\textit{arith} $k$) $\Leftrightarrow$ $a \oslash_{k} b$ is an integer.
\end{corollary}

Consider Example \ref{example:2}: $100=9 \odot_{3} 8$ but $9$ is not a divisor of $100$ (\textit{arith} $3$) because $100 \oslash_{3} 9=100/9+8 \cdot (1-3/2)=64/9 \notin \Z$. That is, $100$ is not the sum of $9$ terms of an arithmetic progression of integers whose difference is $3$.

If we have the divisors of $n$ (arith $k$), then we have the arithmetic progressions whose difference is $k$ with sum $n$: if $d$ is a divisor of $n$ (\textit{arith} $k$), there must exist an integer $a$ such that $n=a \odot_{k} d$. Then, $n=(a-d+1)+ (a-d+1+k)+ \ldots+(a-d+1+(d-1)k)$; hence $n$ is the sum of an arithmetic progression of integers whose difference is $k$. On the other hand, if $n$ is the sum of an arithmetic progression of integers whose difference is $k$, there must exist $a,d \in \Z$, $d>0$ such that $n=a+(a+k)+(a+2k)+\ldots+(a+(d-1)k)$. Then, $n=(a+d-1)\odot_{k} d$ and $d$ is a divisor of $n$ (\textit{arith} $k$).

Now it is clear that the problem to calculate the arithmetic progressions of integers whose difference is $k$ with sum $n$ is equivalent to calculate the set of divisors of $n$ (\textit{arith} $k$).

For the upcoming lemma and the rest of this paper, we use the following notation for even and odd numbers.
\begin{notation} \label{notation:1} We write the set of even and odd numbers as follows: 
\begin{itemize}
	\item $E= \{ \ldots ,-4,-2,0,2,4,6,\ldots \}$.
	\item $O= \{ \ldots ,-3,-1,1,3,5,7,\ldots \}$.
\end{itemize}
\end{notation}

The following lemma characterizes the divisors of $n$ (\textit{arith} $k$). The proof appears in \cite{devega}.

\begin{lemma}
\label{lemma:1}
Given a $k$-\textit{arithmetic} and $a \in \Z$, the divisors of $a \ (arith  \ k)$ are:
\begin{enumerate}
	\item The usual divisors of $a$ if $k \in E$.
	\item The usual divisors of $2a$ except the even usual divisors of $a$ if $k \in O$.
\end{enumerate}
\end{lemma}

\begin{example} \label{example:4}
Express the number $20$ in all possible ways as a sum of an arithmetic progression whose difference is $3$.

\noindent \textit{Solution.} The divisors of $20$ (\textit{arith} $3$) are the usual divisors of $40$ except the even usual divisors of $20$: $\{1, \cancel{2}, \cancel{4}, 5, 8,\cancel{10}, \cancel{20}, 40\}$. We obtain:
\begin{itemize}
\item $d=1\Rightarrow 20\oslash_{3}1=20 \Rightarrow 20=20\odot_{_{3}} 1 \Rightarrow 20=20$.
\item $d=5\Rightarrow 20\oslash_{3}5=2\Rightarrow 20=2\odot_{_{3}} 5 \Rightarrow$ $20=-2+1+4+7+10$.
\item $d=8 \Rightarrow 20\oslash_{3}8=-1\Rightarrow 20=-1\odot_{_{3}}8\Rightarrow 20=-8-5-2+1+4+7+10+13$. 
\item $d=40 \Rightarrow 20\oslash_{3}40=-19 \Rightarrow 20=-19\odot_{_{3}} 40\Rightarrow 20=-58-55- \ldots +56+59$.	
\end{itemize}
\end{example}

\begin{definition}
Given a positive integer $n$ and $k \in \Z$, let $\textrm{D}_{k}(n)$ denote the set of divisors of $n$ (\textit{arith} $k$).
\end{definition}

By Lemma \ref{lemma:1}, we have two options:
\begin{itemize}
	\item If $k \in E$, $\textrm{D}_{k}(n)$ is the set of the usual divisors of $n$. $\textrm{D}_{E}(n)$ denote this case.
	\item If $k \in O$, $\textrm{D}_{k}(n)$ is the set of the usual divisors of $2n$ except the even usual divisors of $n$. $\textrm{D}_{O}(n)$ denote this case.
\end{itemize}

Lemma \ref{lemma:1} clarifies the problem of the representation of a number as the sum of an arithmetic series. We can easily obtain the results previously studied by other authors. For instance, the following corollary appears in \cite{bush}.
\begin{corollary} \label{corollary:2} 
Let $n=2^{e}p_{1}^{e_{1}}\cdots p_{r}^{e_{r}}$ be any positive integer, where $p_{1},\ldots,p_{r}$ are distinct odd primes. The number of different ways in which $n$ can be expressed as the sum of an arithmetic series of integers with a specified odd common difference, $k$, is twice the number of distinct positive odd divisors of $n$.
\end{corollary}
\begin{proof}
Let $\tau_{O}(n)=(e_{1}+1)\cdot \ldots \cdot (e_{r}+1)$ denote the number of odd usual divisors of $n$. Let $\tau_{E}(n)=e\cdot(e_{1}+1)\cdot \ldots \cdot (e_{r}+1)$ denote the number of even usual divisors of $n$. We have to calculate the number of elements of $D_{k}(n)$, $k \in O$, denoted by $|D_{k}(n)|$. By Lemma \ref{lemma:1},
\begin{equation*}
|D_{k}(n)|=\tau(2n)-\tau_{E}(n)=(e+2)\prod_{i=1}^{r}(e_{i}+1) -e\prod_{i=1}^{r}(e_{i}+1)=2 \tau_{O}(n).
\end{equation*}
\noindent Also, if $k \in O$, $|D_{k}(n)|=(\tau_{E}(2n)-\tau_{E}(n))+\tau_{O}(2n)=(\tau_{O}(n))+\tau_{O}(n)$. Hence, exactly half of the elements of $D_{k}(n)$ are even and the other half are odd.
\end{proof}

Let us now study the main result of this paper.

\section{Remarks and examples} \label{sec:3}

If we think about Example \ref{example:4}, then we have a constructive method to solve the main problem of this paper. We are interested in the partitions of a positive integer $n$ in which the nondecreasing sequence of parts forms an arithmetic progression. Therefore, the first term must be a positive integer. Let us consider some remarks.
\begin{remark} \label{remark:1}
The case $k=0$ produces the trivial partitions. There are $\tau(n)$ trivial partitions in this case.
\end{remark}

\begin{remark} \label{remark:2}
The divisor $d=1 \in \textrm{D}_{k}(n)$ always produces the trivial partition $n=n$ for all $k$.
\end{remark}

\begin{remark} \label{remark:3}
If $k \in E$, $k>0$, then we have only to study the divisors $d \in \textrm{D}_{E}(n)$ such that $1<d \leq \sqrt{n}$.
\end{remark}
\begin{proof}
 We are interested in the partitions whose first term is greater than $0$. If $d \in \textrm{D}_{k}(n)$, then we can calculate $a \in \Z$ such that $a\odot_{k}d=n$. By \eqref{eq:2}, $a=n\oslash_{k}d=n/d+(d-1)(1-k/2)$. The first term of the partition is $a-d+1=n/d+(d-1)(1-k/2)-d+1$. If the first term is greater than $0$, then we have the following expression:
	\begin{equation} \label{eq:3}
  \frac{n}{d}+(d-1)(1-\frac{k}{2})-d+1 > 0 \Leftrightarrow k<\frac{2n}{d(d-1)}.
  \end{equation}
\noindent By \eqref{eq:3}, if $k \in E$, $k>0$ and $\frac{2n}{d(d-1)}\leq 2$, then there will be no partition with a positive first term.
Hence, if $d>\sqrt{n}$, then we will not have an element of $\textrm{AP}(n)$.
\end{proof}

\begin{remark} \label{remark:4}
If $k \in O$, $k>0$, then we have only to study the divisors $d \in \textrm{D}_{O}(n)$ such that $1<d < \sqrt{2n}$.
\end{remark}

\begin{proof}
By \eqref{eq:3}, if $k \in O$, $k>0$ and $\frac{2n}{d(d-1)}\leq 1$, then there will be no partition with a positive first term.
Hence, if $d>\sqrt{2n}$, then we will not have an element of $\textrm{AP}(n)$.

If $d=\sqrt{2n}$, then $\sqrt{2n}$ is even and $\sqrt{2n} \mid n$. Thus, by Lemma \ref{lemma:1}, we do not have to consider this case.
\end{proof}

\begin{notation} \label{notation:2}
Let $n$ a positive integer. Let $d \in D_{k}(n)$, $d>1$. Then we denote by $k_{d}$ the critical value $\frac{2n}{d(d-1)}$.
\end{notation}

Let us look at all these remarks with an example:

\begin{example} \label{ex:4}
Calculate $\textrm{AP}(6)$.\\
\textit{Solution.} We are going to consider three cases:
\begin{itemize}
	\item $k=0 \Rightarrow \textrm{D}_{0}(6)=\{1,2,3,6\}$. We have the following possibilities:
	\begin{itemize}
		\item[$\star$] $d=1\Rightarrow 6\oslash_{0}1=6 \Rightarrow 6\odot_{_{0}} 1=6 \Rightarrow 6=6$.
		\item[$\star$] $d=2\Rightarrow 6\oslash_{0}2=4\Rightarrow 4\odot_{_{0}} 2=6 \Rightarrow 3+3=6$.
		\item[$\star$] $d=3\Rightarrow 6\oslash_{0}3=4\Rightarrow 4\odot_{_{0}} 3=6 \Rightarrow 2+2+2=6$.
		\item[$\star$] $d=6\Rightarrow 6\oslash_{0}6=6\Rightarrow 6\odot_{_{0}} 6=6 \Rightarrow 1+1+1+1+1+1=6$.
	\end{itemize}
We have $\tau(6)=4$ trivial partitions in this case. We do not need to repeat this trivial case anymore. Note that the case $k=0$ includes the trivial partition produced by the divisor $d=1$. In the following cases, we will consider the divisors of $\textrm{D}_{k}(6)$ greater than $1$.
\end{itemize}
\begin{itemize}
	\item $k>0, \ k \in E \Rightarrow \ \textrm{D}_{E}(6)=\{1,2,3,6\}$. By Remark \ref{remark:3}, we have only to study the divisors $d \in \textrm{D}_{E}(n)$ such that $1<d \leq \sqrt{6}$, hence we need to study the divisor $d=2$.
\begin{itemize}
\item $d=2 \Rightarrow k_{2}=\frac{2 \cdot 6}{2(2-1)}=6$. By \eqref{eq:3}, the divisor $d=2$ produces partitions of $\textrm{AP}(6)$ in cases such that $k \in E$, $0<k<6$. Hence, $d=2$ produces $2$ partitions ($k=2, \ k=4$).
\item[$\star$] $d=2, \ k=2 \Rightarrow 6\oslash_{2}2=3\Rightarrow 3\odot_{_{2}} 2=6 \Rightarrow 2+4=6$.
\item[$\star$] $d=2, \ k=4, \Rightarrow 6\oslash_{4}2=2\Rightarrow 2\odot_{_{4}} 2=6 \Rightarrow 1+5=6$.
\end{itemize}
\end{itemize}
\begin{itemize}
\item $k>0, \ k \in O \Rightarrow \textrm{D}_{O}(6)=\{1, \cancel{2}, 3, 4, \cancel{6}, 12\}$. By Remark \ref{remark:4}, we have only to study the divisors $d \in \textrm{D}_{O}(n)$ such that $1<d < \sqrt{12}$, hence we need to study the divisor $d=3$.
\begin{itemize}
\item $d=3 \Rightarrow k_{3}=\frac{2 \cdot 6}{3(3-1)}=2$. By \eqref{eq:3}, the divisor $d=3$ produces partitions of $\textrm{AP}(6)$ in cases such that $k \in O$, $0<k<2$. Hence $d=3$ produces $1$ partition ($k=1$).
\item[$\star$] $d=3, \ k=1 \Rightarrow 6\oslash_{1}3=3\Rightarrow 3\odot_{_{1}} 3=6 \Rightarrow 1+2+3=6$.
\end{itemize}
\end{itemize}

\noindent Hence $|\textrm{AP}(6)|=4+2+1=7$.
\end{example}

Let us do a slightly more complicated example. 
\begin{example} Calculate $|\textrm{AP}(100)|$.

\noindent \textit{Solution.} We will write the divisors by pairs. By Remarks \ref{remark:3} and \ref{remark:4}, we will have to study the first row of divisors only.
\begin{itemize}
	\item $k=0$. There are $\tau(100)=9$ trivial partitions. 
\end{itemize}
\begin{itemize}
	\item $k>0$, $k \in E$: $\textrm{D}_{E}(100)= \Big \{ \begin{array}{lllll} 1 & 2 & 4 & 5 & 10 \\ 
100 & 50 & 25 & 20 &  \end{array} \Big \}$.

\medskip
\begin{tabular}{c|c|c|c}
$d$ & $k_{d}=\frac{2\cdot 100}{d\cdot(d-1)}$ & $k \in E$ and $0<k<k_{d}$ & Number of $\textrm{AP}$ partitions \\ \hline
$2$ & $100$ & $k=2$, $k=4$, \ldots, $k=98$ & $49$ \\
$4$ & $16.\hat{6}$ & $k=2$, $k=4$, \ldots, $k=16$ & $8$ \\
$5$ & $10$ & $k=2$, $k=4$, $k=6$, $k=8$ & $4$ \\
$10$ & $2.\hat{2}$ & $k=2$ & $1$ \\ \hline
\end{tabular}
\end{itemize}

\begin{itemize}
\item $k>0$, $k \in O$: $\textrm{D}_{O}(100)= \Big \{ \begin{array}{llllll} 1 & \cancel{2} & \cancel{4} & 5 & 8 & \cancel{10} \\ 
200 & \cancel{100} & \cancel{50} & 40 & 25 & \cancel{20}  \end{array} \Big \}$.

\medskip
\begin{tabular}{c|c|c|c}
$d$ & $k_{d}=\frac{2\cdot 100}{d\cdot(d-1)}$ & $k \in O$ and $0<k<k_{d}$ & Number of $\textrm{AP}$ partitions \\ \hline
$5$ & $10$ & $k=1$, $k=3$, \ldots, $k=9$ & $5$ \\
$8$ & $3.57\ldots$ & $k=1$, $k=3$ & 2\\ \hline
\end{tabular}
\end{itemize}

\medskip
\noindent Hence $|\textrm{AP}(100)|=9+49+8+4+1+5+2=78$.

If we want to calculate a concrete partition, for instance $d=5$, $k=7$, then we can do: $100 \oslash_{7} 5=100/5+4(1-7/2)=10 \Rightarrow100=10 \odot_{7} 5=6+13+20+27+34$. 
\end{example}

We can use the floor and the ceiling functions to count the even and odd numbers in each case.
\begin{remark} \label{remark:5}
The number of positive even numbers less than a real $x>0$ is given by the expression $\lfloor \frac{1}{2} (\lceil x \rceil -1)\rfloor$,
where $\lfloor x \rfloor$ is the greatest integer $\leq x$ and $\lceil x \rceil$ is the smallest integer $\geq x$.
\end{remark}

\begin{remark} \label{remark:6}
The number of positive odd numbers less than a real $x>0$ is given by the expression $\lfloor \frac{1}{2} \lceil x \rceil\rfloor$.
\end{remark}

With all of the above, we can prove Theorem \ref{theorem:1}.

\section{Proof of Theorem \ref{theorem:1}} \label{sec:4}
Let us summarize the method explained in the previous section. Then we have to study three cases to calculate $|\textrm{AP}(n)|$:
\begin{itemize}
	\item Case $k=0$: there are $\tau(n)$ trivial partitions (Remark \ref{remark:1}). The divisor $d=1$ always produces the trivial partition $n=n$ (Remark \ref{remark:2}). This partition is counted in this case only. In the following cases, we will consider the divisors of $\textrm{D}_{k}(n)$ greater than 1.
	
	\medskip
	\item Case $k>0$, $k \in E$: we have only to study the divisors $d \in \textrm{D}_{E}(n)$ such that $1<d\leq \sqrt{n}$ (Remark \ref{remark:3}). By \eqref{eq:3}, each divisor produces partitions in the cases such that $k \in E$, $0<k<k_{d}$. By Remark \ref{remark:5}, there are $\lfloor \frac{1}{2} (\lceil k_{d} \rceil -1)\rfloor$ partitions of $\textrm{AP}(n)$ in this case.
	
	\medskip
	\item Case $k>0$, $k \in O$: we have only to study the divisors $d \in \textrm{D}_{O}(n)$ such that $1<d< \sqrt{2n}$ (Remark \ref{remark:4}). By \eqref{eq:3}, each divisor produces partitions in the cases such that $k \in O$, $0<k<k_{d}$. By Remark \ref{remark:6}, there are $\lfloor \frac{1}{2} \lceil k_{d} \rceil\rfloor$ partitions of $\textrm{AP}(n)$ in this case.
\end{itemize}
Then,
\begin{equation*}
|\textrm{AP}(n)|=\tau(n)+\sum_{\scriptstyle  d \ \in \ \textrm{D}_{\textrm{\tiny E}}(n) \atop \scriptstyle  1 \ < \ d \ \leq \ \sqrt{n}}\left\lfloor \frac{1}{2} (\left\lceil k_{d} \right\rceil -1)\right\rfloor + \sum_{\scriptstyle  d \ \in \ \textrm{D}_{ \textrm{\tiny O}}(n) \atop \scriptstyle 1 \ < \ d \ < \ \sqrt{2n}}\left\lfloor \frac{1}{2} \left\lceil k_{d} \right\rceil \right\rfloor. \tag*{\qed}
\end{equation*}

Once the problem is understood and solved, the only difficulty in calculating $\textrm{AP}(n)$ is to obtain the set of divisors of $2n$.

Figure \ref{fig:1} looks like the famous Goldbach's comet.

\begin{figure}[ht]
\centering
\includegraphics[width=0.7\textwidth]{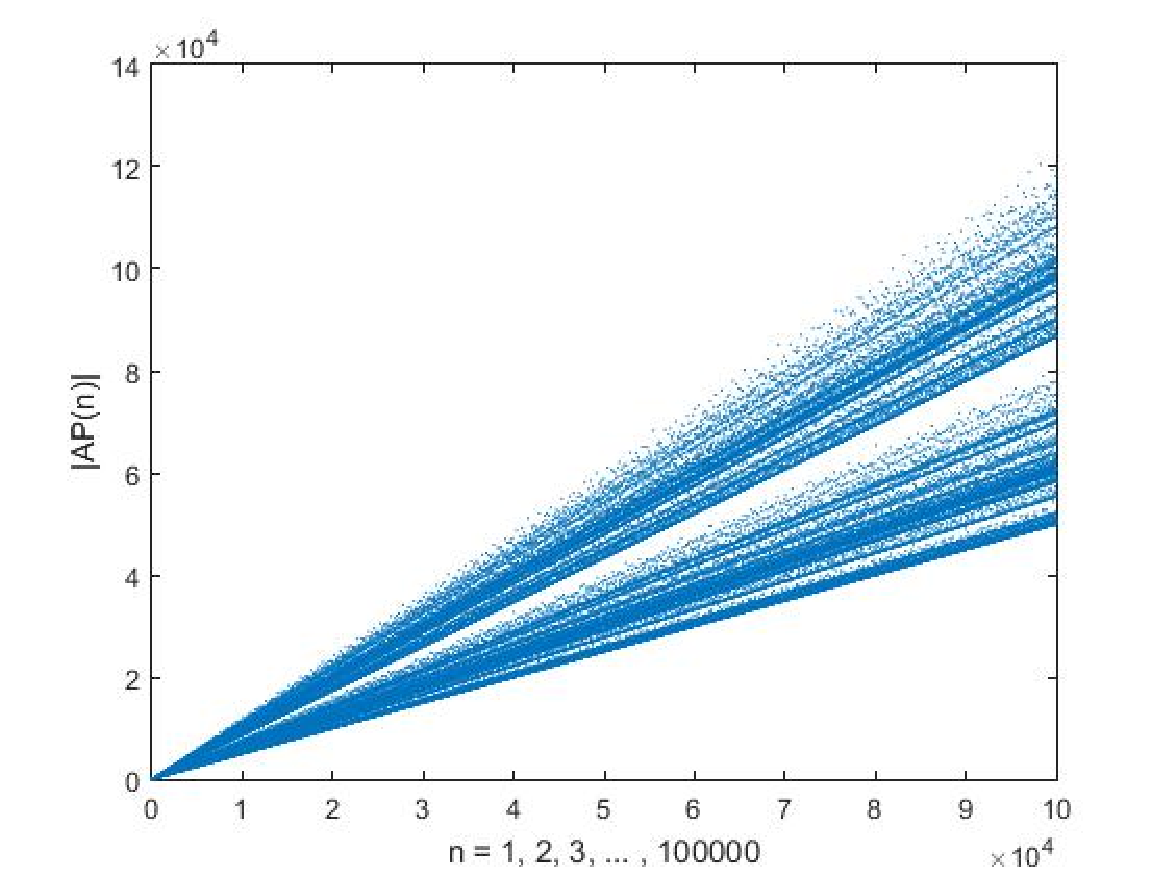}
\caption{$|\textrm{AP}(n)|$, $n=1,2,\ldots,100000$}
\label{fig:1}       
\end{figure}

\section{On the lengths of the partitions of \textrm{AP}(n)} \label{sec:5}
A question proposed in \cite{munagi1,munagi2} deals with the different lengths of the partitions of $\textrm{AP}(n)$. If we consider a partition of $\textrm{AP}(n)$ as a $d$-tuple $(n_{1},n_{2},\ldots,n_{d})$, then we can define the set $\textrm{APdiv}(n)$ as the different lengths $d$ of the partitions of $\textrm{AP}(n)$. Since the trivial partitions have lengths equal to the divisors of $n$, $\textrm{D}_{E}(n) \subseteq \textrm{APdiv}(n)$. By Theorem \ref{theorem:1}, the different lengths will be the elements of $\textrm{D}_{E}(n)$ (usual divisors) and the even elements of $\textrm{D}_{O}(n)$ that produce partitions of $\textrm{AP}(n)$.

\begin{corollary} \label{corollary:3}
$\displaystyle |\textrm{APdiv}(n)|=\tau(n) + \sum_{\scriptstyle  d \ \in \ E \cap \textrm{D}_{ \textrm{\tiny O}}(n) \atop \scriptstyle 1 \ < \ d \ < \ \sqrt{2n}}1$.
\end{corollary}
\begin{proof}
By Theorem \ref{theorem:1}, 
\begin{equation*}
\textrm{APdiv}(n)=\textrm{D}_{E}(n) \cup  \{ d \in E\cap \textrm{D}_{O}(n): 1<d<\sqrt{2n} \},
\end{equation*}
and the result follows.
\end{proof}

\begin{example} \label{example:7}
Calculate $|\textrm{APdiv}(500)|$.\\
\textit{Solution.} Since $500=2^{2}\cdot 5^{3}$, $\tau(500)=3\cdot 4=12$.

\begin{equation*}
\textrm{D}_{O}(500)= \Big \{ \begin{array}{llllllll} 1 & \cancel{2} & \cancel{4} & 5 & 8 & \cancel{10} & \cancel{20} & 25 \\ 1000 & \cancel{500} & \cancel{250} & 200 & 125 & \cancel{100}  & \cancel{50}  & 40 \end{array} \Big \}.
\end{equation*}

\noindent Then, $\textrm{APdiv}(500)=\textrm{D}_{E}(500) \cup \{ 8 \}$ and $|\textrm{APdiv}(500)|=12+1=13$.
\end{example}

The sequence $(|\textrm{APdiv}(n)|)_{n>0}$ occurs as sequence number \seqnum{A175239} in the Online Encyclopedia of Integer Sequences \cite{oeis}. 

\section{Conclusion} \label{sec:6}
The novel way of studying a partition problem by calculating the divisors of a number in an arithmetic similar to the usual one is the main contribution of this paper. The study of the arithmetic generated by $\odot_{k}$ ($k$-\textit{arithmetic}) is interesting by itself. An improvement of \cite{devega} proposes to study the arithmetic generated by any integer sequence $(a_{n})_{n>0}$. Try to convert a partition problem to a divisors problem in an arithmetic generated by an integer sequence is a topic that needs more work.

\subsection*{Acknowledgements}
This work was supported by King Juan Carlos University under grant C2PREDOC2020.

\end{document}